\documentclass[11 pt]{amsart}

\usepackage{pdfsync}
\usepackage{amscd}
\usepackage{amsmath}
\usepackage{amssymb}
\usepackage{amsthm}
\usepackage{epsf}
\usepackage{latexsym}
\usepackage{verbatim}
\usepackage[utf8]{inputenc}
\usepackage{pdfsync}
\input epsf.tex

\usepackage{enumerate}

\newtheorem{theorem}{Theorem}[section]
\newtheorem{lemma}[theorem]{Lemma}
\newtheorem{corollary}[theorem]{Corollary}

\begin{document}






\title[On Rationally Ergodic and Rationally Weakly Mixing]{On Rationally Ergodic and Rationally Weakly Mixing Rank-One Transformations}

\author[Dai]{Irving Dai}
\address[Irving Dai]{Harvard College\\ University Hall \\
Cambridge, MA 02138, USA}
\email{ifdai@college.harvard.edu}

\author[Garcia]{Xavier Garcia}
\address[Xavier Garcia]{University of Minnesota \\
Minneapolis, MN 55455-0213, USA}
\email{garci363@umn.edu}

\author[P\u{a}durariu]{Tudor P\u{a}durariu}
\address[Tudor P\u{a}durariu]{University of California, Los Angeles, CA 90095-1555, US }
\email {tudor\_pad@yahoo.com}

\author[Silva]{Cesar E. Silva}
\address[Cesar E. Silva]{Department of Mathematics\\
     Williams College \\ Williamstown, MA 01267, USA}
\email{csilva@williams.edu}


\begin{abstract}  
We study the notions of  weak rational ergodicity and rational weak mixing as defined by Jon Aaronson. We prove that various families of infinite measure-preserving rank-one transformations possess (or do not posses) these properties, and consider their relation to other notions of mixing in infinite measure. 
\end{abstract}

\maketitle

{\allowdisplaybreaks

\section{Definitions and Preliminaries}

\noindent
Let $(X,\mathcal{B},\mu)$ be a standard Borel  measure space with a $\sigma$-finite nonatomic  measure $\mu$. In most cases, we will assume that $\mu$ is infinite. A transformation $T: X \rightarrow X$ is {\bf measurable} if $T^{-1}A\in\mathcal B$ for all $A\in\mathcal B$. A measurable transformation $T$ is 
{\bf measure-preserving} if  $\mu(A)=\mu(T^{-1}A)$ for all $A\in\mathcal B$. 
We say that $T$ is {\bf ergodic} if every $T$-invariant set (i.e, $T^{-1}A=A $ mod $\mu$) is null ($\mu(A)=0$) or full $(\mu(X\setminus A)=0$). We say that $T$ is {\bf conservative} if for every measurable set $A$ of positive measure, there exists a positive integer  $n$ such that $\mu(A \cap T^{-n}A)>0$. 
 It follows that $T$ is conservative and ergodic if and only if for every set $A$ of positive measure, 
  $\bigcup_{n=0}^\infty T^{-n}A=X$ mod $\mu$.
An {\bf invertible measurable  transformation} is a measurable transformation whose inverse is also measurable.  Throughout this paper, we will assume that $T$ is an invertible, conservative ergodic, measure-preserving transformation on  $(X,\mathcal{B},\mu)$, and we will typically use the forward images  $T^nA$ instead of $T^{-n}A$. \\
\\
When $T$ is  a measure-preserving transformation on   a probability space $X$, the Birkhoff ergodic theorem states   that  ergodicity is equivalent to having the convergence
\begin{align}\label{birk}
\dfrac{1}{n} \displaystyle \sum_{k = 0}^{n-1} \mu(A \cap T^k B) \rightarrow \mu(A)\mu(B)
\end{align}
\noindent
for all measurable $A, B \subset X$.
This gives a quantitative estimate for the average number of visits of one set to another.
 When $X$ has infinite measure, however, the Birkhoff ergodic theorem 
implies that the Cesaro averages of \eqref{birk} converge to $0$ for all pairs $A, B$ of finite measure.
Moreover, in \cite{Aa77} Aaronson proved that there exists  no sequence of normalizing constants for which the averages of
\eqref{birk} converge to $\mu(A)\mu(B)$, and he proposed in turn the definitions of rational ergodicity and weak rational ergodicity. \\

\noindent
For any measurable set $F \subset X$ of finite positive measure, define the $\mathbf{intrinsic}$ $\mathbf{weight}$ $\mathbf{sequence}$ of $F$ to be

\[
u_n(F) = \dfrac{\mu(F \cap T^nF)}{\mu(F)^2}
\]
\noindent
and write 

\[
a_n(F) = \displaystyle \sum_{k = 0}^{n-1} u_k(F).
\]

\noindent
A transformation $T$ is said to be {\bf weakly rationally ergodic} (see \cite{Aa77}) if there exists a measurable set $F \subset X$ of positive finite measure such that for all measurable $A$, $B$ $\subset F$, we have

\begin{equation}\label{e:raterg}
\dfrac{1}{a_n(F)} \displaystyle \sum_{k = 0}^{n-1} \mu(A \cap T^kB) \rightarrow \mu(A)\mu(B)
\end{equation}

\noindent
as $n \rightarrow \infty$. If this convergence happens only along a subsequence $\{n_i\}$ of $\mathbb{N}$, we say that $T$ is $\mathbf{subsequence}$ $\mathbf{weakly}$ $\mathbf{rationally}$ $\mathbf{ergodic}$. To emphasize the set $F$, we will sometimes say $T$ is $\mathbf{weakly}$ $\mathbf{rationally}$ $\mathbf{ergodic}$ $\mathbf{on \ F}$. Note that any measure-preserving ergodic transformation on a probability space is trivially weakly rationally ergodic, by taking $F$ to be the whole space itself. Then $a_n(F) = n$, so \eqref{e:raterg}  reduces to the Cesaro sum definition of ergodicity. \\

\noindent
A transformation  $T$ is said to be $\mathbf{(spectrally)}$ $\mathbf{weakly}$ $\mathbf{mixing}$ if whenever $f\in L^\infty(X,\mu)$ and $f\circ T= z f $  for some $z\in \mathbb{C}$, then $f$ is constant a.e. When  $X$ is a probability space, this is equivalent to ergodicity of the Cartesian square and also  to  the strong Cesaro convergence
\[
\dfrac{1}{n} \displaystyle \sum_{k = 0}^{n-1} |\mu(A \cap T^k B) - \mu(A)\mu(B)| \rightarrow 0
\]
for all measurable $A, B \subset X$.  In \cite{AaLiWe79}, it was shown that for infinite measure-preserving transformations, (spectral) weak mixing 
is strictly weaker than ergodicity of the Cartesian square. \\
\\
\noindent
Another property we consider that is equivalent to weak mixing in the finite measure-preserving case is 
double ergodicity.  This property was introduced by Furstenberg in \cite{Fu81} and was shown to be 
equivalent to weak mixing for probability-preserving transformations, but was not given a specific name. A  transformation $T$ is said to be {\bf doubly ergodic} if for every pair of sets $A$ and $B$ with positive measure, there exists a positive integer  $n$ for which $\mu(A \cap T^nA)$ and $\mu(B \cap T^nA)$ are simultaneously nonzero. In the infinite measure-preserving case, double ergodicity is strictly stronger than spectral weak mixing and is properly implied by ergodic Cartesian square \cite{BoFiMaSi01}. \\
\\
\noindent
More recently, Aaronson introduced another notion of weak mixing for infinite measure that generalizes
rational ergodicity.
A transformation $T$ is said be  {\bf rationally weakly mixing} (see \cite{Aa12}) if there exists a measurable set $F \subset X$ of positive finite measure such that for all measurable $A$, $B$ $\subset F$, we have
\begin{equation}\label{e:rwm}
\dfrac{1}{a_n(F)} \displaystyle \sum_{k = 0}^{n-1} |\mu(A \cap T^kB) - \mu(A)\mu(B)u_k(F)| \rightarrow 0
\end{equation}

\noindent
as $n \rightarrow \infty$. Again, it is clear that rational weak mixing reduces to the usual definition of weak mixing in the finite measure-preserving case. \\
\\
\noindent
We now describe our main results. In Section~\ref{S:ratergodicity} we prove that a large class
of rank-one transformations are weakly rationally ergodic and discuss the notions of rational
ergodicity and bounded rational ergodicity in this context. In Section~\ref{S:ratweakmix}  
we construct  a class of rank-one transformations that are not rationally weakly mixing; in particular, we obtain a transformation which is rationally ergodic and spectrally weakly mixing but not 
rationally weakly mixing. This negatively answers a question of Aaronson's. (After this work was completed, we learned that Aaronson had also independently answered this question \cite{Aa12b}.)
Section~\ref{S:doubleerg} shows that rational weak mixing implies double ergodicity and 
constructs a transformation that is not rationally weakly mixing and which we conjecture to be doubly ergodic.
Section~\ref{S:zerotype} proves that the notion of zero-type for infinite measure-preserving transformations (whose spectral definition is similar to the  mixing  condition in the case of 
probability-preserving transformations) is independent of rational weak mixing. 
Finally, in Section~\ref{S:ratweakmixex} we present a class of rank-one transformations that are 
rationally weakly mixing. As remarked in \cite{Aa12}, all the examples of rationally weakly mixing transformations
constructed in \cite{Aa12} are of the type $T\times S$, where $T$ is an infinite measure-preserving 
$K$-automorphism and $S$ is a mildly mixing probability-preserving transformation. These examples
have countable Lebesgue spectrum and are of a different nature than our rank-one constructions.

\subsection{Acknowledgements}
This paper was based on research done by the Ergodic Theory group of the 2012 SMALL Undergraduate Research Project at Williams College. Support for this project was provided by the National Science Foundation REU Grant DMS - 0353634 and the Bronfman Science Center of Williams College.  We are indebted to Jon Aaronson for conversations and suggestions during discussions of our work at the 2012 Williams Ergodic Theory Conference. We would also like to acknowledge the other members of the 2012 Ergodic Theory group: Shelby Heinecke, Emily Wickstrom, and Evangelie Zachos. We would like to thank the referee for comments that improved the paper.

\subsection{Rank-One Transformations (Basics)}

We briefly review (rank-one) cutting-and-stacking transformations (see e.g. \cite{Si08}). A $\mathbf{column}$ is an ordered collection of pairwise disjoint intervals (called $\mathbf{levels}$) in $\mathbb{R}$, each of the same measure. We think of the levels in a column as being stacked on top of each other, so that the $(j+1)$-st level is directly above the $j$-th level. Every column $C = \{J_j\}$ is associated with a natural column map $T_C$ sending each point in $J_j$ to the point directly above it in $J_{j+1}$. (Note that $T_C$ is undefined on the top level of $C$.) A $(\mathbf{rank}$-$\mathbf{one})$ $\mathbf{cutting}$-$\mathbf{and}$-$\mathbf{stacking}$ construction for $T$ consists of a sequence of columns $C_n$ such that:

\begin{enumerate}[(a)]
\item The first column $C_0$ is the unit interval.
\item Each column $C_{n+1}$ is obtained from $C_n$ by cutting $C_n$ into $r_n\geq 2$ subcolumns of equal width, adding any number of new levels (called $\mathbf{spacers}$) above each subcolumn, and stacking every subcolumn under the subcolumn to its right. In this way, $C_{n+1}$ consists of $r_n$ copies of $C_n$, possibly separated by spacers.
\item The collection of levels $\displaystyle \bigcup_n C_n$ forms a generating semiring for $\mathcal{B}$.
\end{enumerate}

\noindent
Observing that $T_{C_{n+1}}$ agrees with $T_{C_n}$ everywhere where $T_{C_n}$ is defined, we then take $T$ to be the limit of $T_{C_n}$ as $n \rightarrow \infty$. 

\subsection{Rank-One Transformations (Notation)}

Let $T$ be a rank-one transformation, and fix any column $C_n$ of $T$. We denote the number of levels in $C_n$ by $h_n$ and write $w_n$ for the width of each level. We denote the height of any level $J$ in $C_n$ by $h(J)$, with the convention that $0 \leq h(J) < h_n$. For each $0 \leq k < r_n$, let $s_{n,k}$ be the number of spacers added above the $k$-th subcolumn of $C_n$, and denote the number of levels in the $k$-th subcolumn (after adding spacers) by $h_{n, k} = h_n + s_{n, k}$. \\
\\
Define $T$ to be $\mathbf{normal}$ if $s_{n, r_n-1} > 0$ for infinitely many values of $n$. (This means that at least one spacer is added above the rightmost subcolumn infinitely many times.) In addition, we say that $T$ has a $\mathbf{bounded \ number \ of}$ $\mathbf{cuts}$ if $\sup \{r_n\} < \infty$; this implies that $T$ is partially rigid and of infinite conservative index \cite{AdFrSi97}. \\
\\
Given any level $J$ from $C_n$ and any column $C_m$ of $T$ with $m \geq n$, we define the $\mathbf{descendants}$ of $J$ in $C_m$ to be the collection of levels in $C_m$ whose disjoint union is $J$. We denote this set by $D(J, m)$. Occasionally, we will also use $D(J, m)$ to refer to the heights of the descendants of $J$ in $C_m$. In the case when $J$ is the unit interval $I$, for each $m \in \mathbb{N}$ it will be convenient to define $M_m = \max(D(I, m))$. (That is, $M_m$ is the height of the uppermost descendant of $I$ in $C_m$.)  \\
\\
We say that $T$ $\mathbf{grows \ exponentially}$ if $2s_{n,r_n-1} \geq h_{n+1}$ for every $n$. Intuitively, this means that the upper half of every column $C_n$ consists of spacers added during the $(n-1)$-st stage of construction. In particular, the descendants of any level $J$ from an earlier column must lie in the lower half of $C_n$. Note that any $T$ which grows exponentially is clearly normal. 


\section{Rational Ergodicity}\label{S:ratergodicity}
\noindent
In this section, we establish some introductory ideas and prove that a large class of rank-one transformations are rationally ergodic. \\
\\
We begin with a computational lemma. Suppose that $T$ is a normal rank-one transformation. Then we claim that the partial sums $a_n(J)$ for any level $J$ can be computed from the descendant heights $D(J,N)$ for $N$ sufficiently large. More precisely,

\begin{lemma}\label{L:lem2.1}
Let $T$ be a normal rank-one transformation. Fix any level $J$ and $n \in \mathbb{N}$. Then for every $N$ sufficiently large, we have
\[
\mu(J \cap T^kJ) = w_N\cdot  |D(J, N) \cap (k + D(J, N))|
\]
for all $0 \leq k < n$. Consequently,
\[
\displaystyle \sum_{k = 0}^{n-1}  \mu(J \cap T^kJ) = w_N\cdot \sum_{k = 0}^{n-1}  |D(J, N) \cap (k + D(J, N))|.
\]
\end{lemma} 
\begin{proof}
Fix any level $J$, and let $n \in \mathbb{N}$ be arbitrary. Since $T$ is normal, we can find some column $C_N$ in which all the heights $D(J, N)$ are at most $h_N - n$. For any $0 \leq k < n$ and level $J_i \in D(J, N)$, the image $T^k(J_i)$ is then the level in $C_N$ of height $h(J_i) + k$. The conclusion follows immediately.
\end{proof}

\noindent
We will sometimes need to compute $\mu(J \cap T^kJ)$ for $k < 0$. For this, simply observe that 
\[
\mu(J \cap T^kJ) = \mu(T^{-k}J \cap J)
\] 
and 
\[
|D(J, N) \cap (k + D(J, N))| = |(-k + D(J, N)) \cap D(J, N)|,
\]
so in fact Lemma~\ref{L:lem2.1} holds for all $-n < k < n$. \\
\\
We thus calculate $D(J, N)$. Suppose that $J$ is a level in $C_j$ of height $h(J)$. Then $J$ splits into $r_j$ levels in $C_{j+1}$ of heights 
\[
\{h(J)\} \cup \{h(J) + \displaystyle \sum_{k = 0}^i h_{j, k}: 0 \leq i < r_j - 1\}.
\]
Letting
\[
H_j = \{0\} \cup \left\{\sum_{k = 0}^i h_{j, k} : 0 \leq i < r_j - 1\right\},
\]
it follows inductively that 
\[
D(J, N) = h(J) + H_j \oplus H_{j+1} \oplus \cdots \oplus H_{N-1}.
\]

\noindent
\\
We now show that every normal rank-one transformation satisfies condition \eqref{e:raterg} for $A$, $B$ finite unions of levels and $F$ the unit interval. In this context, we note that Aaronson \cite[Theorem 6.1]{Aa77} has shown every set of finite measure $F$  contains a dense algebra of sets satisfying \eqref{e:raterg}, but at the same time it is never true that \eqref{e:raterg} is satisfied for all measurable sets
in every set $F$ of finite positive measure \cite[Theorem 6.2]{Aa77}.


\begin{theorem}\label{T:normal}
Let $T$ be a normal rank-one transformation. Then $T$ satisfies condition \eqref{e:raterg} for $A$, $B$ finite unions of levels and $F$ the unit interval.
\end{theorem}
\begin{proof}
Let $F = I$ denote the unit interval. We begin by proving \eqref{e:raterg} for $A = B = J$, where $J$ is the bottom level of any column $C_j$. We need to show that
\[
\dfrac{1}{a_n(I)} \displaystyle \sum_{k = 0}^{n-1} \mu(J \cap T^{k}J) \rightarrow \mu(J)^2
\]
as $n \rightarrow \infty$. For $N$ sufficiently large (as a function of $n$), we have
\[
\displaystyle \sum_{k = 0}^{n-1} \mu(J \cap T^{k}J) = w_N \left( \sum\limits_{k = 0}^{n-1}|D(J,N)\cap (k + D(J,N))| \right)
\]
by Lemma~\ref{L:lem2.1}. Now, writing
\[
D(I, N) = H_0 \oplus H_1 \oplus \cdots \oplus H_{N-1}
\]
and
\[
D(J, N) = H_j \oplus H_{j+1} \oplus \cdots \oplus H_{N-1},
\]
we may express $D(I, N) = A \oplus B$ and $D(J, N) = B$ with $A = H_0 \oplus H_1 \oplus \cdots \oplus H_{j-1}$. Noting that $\mu(J) = 1/|D(I, j)| = 1/|A|$, we thus wish to show
\[
\displaystyle \dfrac{w_N}{a_n(I)} \left(|A|^2\sum\limits_{k = 0}^{n-1}|B \cap (k + B)|\right) \rightarrow 1.
\]
\noindent
We give the term inside the parentheses a combinatorial interpretation. Let $P(n)$ denote the number of ordered quadruplets $(a, a', b, b')$ with $a, a' \in A$ and $b, b' \in B$ for which $0 \leq b - b' < n$. Then the above quotient is precisely $w_NP(n) / a_n(I)$, since $|B \cap (k + B)|$ counts the number of pairs $b, b' \in B$ with $k = b - b'$. \\
\\
Now let $M$ be the maximum value of $A - A$. We claim that the following inequality holds:
\[
\displaystyle \sum_{k = M}^{n - 1 - M} |(A \oplus B) \cap (k + A \oplus B)| \leq P(n) \leq  \sum_{k = -M}^{n - 1 + M} |(A \oplus B) \cap (k + A \oplus B)|.
\]
\noindent
Indeed, the sum on the left counts the number of quadruplets $(a, a', b, b')$ with $M \leq a - a' + b - b' < n - M$; the sum on the right counts the number of quadruplets with $-M \leq a - a' + b - b' < n + M$. Clearly, any quadruplet with $M \leq a - a' + b - b' < n - M$ has $0 \leq b - b' < n$. Similarly, any quadruplet with $0 \leq b - b' < n$ has $-M \leq a - a' + b - b' < n + M$. Recalling that $A \oplus B = D(I, N)$, it thus follows that
\[
\dfrac{1}{a_n(I)} \sum_{k = M}^{n - 1 - M} \mu(I \cap T^kI) \leq \dfrac{w_N}{a_n(I)} (P(n)) \leq \dfrac{1}{a_n(I)} \sum_{k = -M}^{n - 1 + M} \mu(I \cap T^kI).
\]
\noindent
Now, $M$ is a fixed constant, independent of $n$. Furthermore, the sequence $\mu(I \cap T^k(I))$ is bounded above by 1 but has divergent sum. Hence both sides of the above inequality tend to 1 as $n \rightarrow \infty$, showing that $w_NP(n) / a_n(I) \rightarrow 1$, as desired. This proves \eqref{e:raterg} for $A = B = J$ , where $J$ is the bottom level of any column. \\
\\
We now prove \eqref{e:raterg} for $J$ and $J'$ any two levels in the same column. By applying $T^{-1}$ and using the fact that $T$ is measure-preservin, we may assume that one of the two levels (say $J$) is actually the bottom level of the column. Letting $J' = T^d(J)$ for some $d$, we wish to show
\[
\dfrac{1}{a_n(I)} \sum_{k = 0}^{n-1} \mu(J \cap T^{k + d}J) \rightarrow \mu(J)^2.
\]
Now, we have from before that
\[
\dfrac{1}{a_n(I)} \sum_{k = 0}^{n-1} \mu(J \cap T^{k}J) \rightarrow \mu(J)^2.
\]
Since $\mu(J \cap T^{k}J)$ is bounded and $a_n(I) \rightarrow \infty$, the conclusion follows immediately. \\
\\
Finally, we extend to finite unions of levels. Without loss of generality, we may assume that $J$ and $J'$ are both disjoint unions of images of the same level $K$. The desired statement then follows from summing together the limits \eqref{e:raterg} for each pair of images.
\end{proof}


\noindent
We now show that under certain conditions, we can extend the results of Theorem~\ref{T:normal} to all sets $A$ and $B$ (thus proving weak rational ergodicity).

\begin{theorem}\label{T:wraterg}
Let $T$ be an exponentially growing rank-one transformation with a bounded number of cuts. Then $T$ is weakly rationally ergodic.
\end{theorem}

\begin{proof}
We show that for $T$ satisfying the above hypotheses, it suffices to prove \eqref{e:raterg} for finite unions of levels (as in Theorem~\ref{T:normal}). Indeed, given arbitrary measurable sets $A, B \subset I$, choose $D \subset I$ a finite union of levels for which $\mu(D \triangle B) < \varepsilon$. We claim that there is some constant $c$ such that
\begin{equation}
\left| \dfrac{1}{a_n(I)} \sum\limits_{k = 0}^{n-1}\mu(A\cap T^kB) - \dfrac{1}{a_n(I)}\sum_{k = 0}^{n-1}\mu(A\cap T^kD)\right| \leq c\varepsilon
\end{equation}
\\
\noindent
for every $n$. Indeed, let $B_0 = B \cap D$, and write $B = B_0 \cup B_1$ and $D = B_0 \cup D_1$. Then the above difference reduces to
\[
\left| \dfrac{1}{a_n(I)} \sum\limits_{k = 0}^{n-1}\mu(A\cap T^kB_1) - \dfrac{1}{a_n(I)}\sum_{k = 0}^{n-1}\mu(A\cap T^kD_1)\right|.
\]
\noindent
Now, we claim that we can bound
\begin{equation}
\displaystyle \dfrac{1}{a_n(I)} \sum\limits_{k = 0}^{n-1}\mu(A\cap T^kB_1) \leq c \mu(B_1)
\end{equation}
for some $c$ independent of $n, A,$ and $B_1$. Applying this bound with $D_1$ in place of $B_1$ will bound (4) by $c(\mu(B_1)+\mu(D_1)) \leq 2c\varepsilon$, as desired. \\
\\
Recall that
 $M_m = \max(D(I, m))$, for $m\in\mathbb{N}$.  Clearly, $\{M_m\}$ is an increasing sequence. For any fixed $n$, if we choose $m$ such that $M_{m-1} \leq n-1 < M_m$, we have
\[
\displaystyle \sum_{k = 0}^{n-1} \mu(A \cap T^kB_1) \leq \sum_{k = 0}^{n-1} \mu(I \cap T^kB_1) \leq \sum_{k = 0}^{M_m} \mu(I \cap T^kB_1)
\]
and
\[
\displaystyle \sum_{k = 0}^{M_{m-1}} \mu(I \cap T^kI) \leq \sum_{k = 0}^{n-1} \mu(I \cap T^kI) = a_n(I).
\]
To prove (5), it thus suffices to find some $c$ such that
\begin{equation}
\sum_{k = 0}^{M_m} \mu(I \cap T^kB_1) \leq c\mu(B_1) \sum_{k = 0}^{M_{m-1}} \mu(I \cap T^kI)
\end{equation}
for every $m$. \\
\\
Now observe that the sets $T^kI$ with $-M_m \leq k \leq M_m$ cover each point of $I$ exactly $|D(I, m)|$ times. Indeed, consider the column $C_m$ and fix any $x \in I$. Let $x$ be contained in $J$, where $J$ is some level from $D(I, m)$. For any level $J'$ in $D(I, m)$, we claim that there is exactly one value of $k$ between $-M_m$ and $M_m$ for which $T^kJ' \cap J \neq \emptyset$. Indeed, suppose $0 \leq k \leq M_m$ and $T^kJ' \cap J \neq \emptyset$. Any forward image $T^kJ'$ with $0 \leq k \leq M_m$ is just a translation upwards by $k$ levels, since $h_m \geq 2M_m$. (This is implied by our hypothesis that $T$ is exponentially growing.) Hence in this case $k$ must equal $h(J) - h(J')$. On the other hand, suppose $-M_m \leq k < 0$ and $T^kJ' \cap J \neq \emptyset$. Then $J' \cap T^{-k}J \neq \emptyset$, and exactly the same argument shows that $-k = h(J') - h(J)$ (i.e., $k = h(J) - h(J')$). The claim is then immediate. \\
\\
We thus have
\[
\sum_{k = -M_m}^{M_m} \mu(I \cap T^kB_1) = \sum_{k = -M_m}^{M_m} \mu(T^kI \cap B_1) = |D(I, m)| \mu(B_1)
\]
and
\[
\sum_{k = -M_{m-1}}^{M_{m-1}} \mu(I \cap T^kI) = |D(I, m-1)|.
\]
Hence
\[
\sum_{k = -M_m}^{M_m} \mu(I \cap T^kB_1) = \left(\dfrac{|D(I,m)|}{|D(I,m-1)|}\right)\mu(B_1)\left(\sum_{k = -M_{m-1}}^{M_{m-1}}\mu(I \cap T^kI)\right)
\]
and so
\[
\sum_{k = 0}^{M_m} \mu(I \cap T^kB_1) \leq \left(\dfrac{|D(I,m)|}{|D(I,m-1)|}\right)\mu(B_1)\left(2\left(\sum_{k = 0}^{M_{m-1}}\mu(I \cap T^kI)\right) - 1\right).
\]
\\
\noindent
But $|D(I,m)|/|D(I, m-1)| = r_{m-1}$, and $T$ has a bounded number of cuts. We thus easily obtain (6). Hence (4) holds, and we can approximate $B$ with $D$ a finite union of levels. Applying a similar argument to $A$ shows that it suffices to prove \eqref{e:raterg} for all $A$, $B$ finite unions of levels, which is the content of Theorem~\ref{T:normal}.
\end{proof}

\noindent
We now consider some alternate notions of rational ergodicity, also due to Aaronson  \cite{Aa77}. For any measurable function $f$, recall the notation
\[
S_n(f) = \displaystyle \sum_{k=0}^{n-1}  f \circ T^k.
\]
\noindent
We say that $T$ is $\mathbf{rationally \ ergodic}$  if there exists a set $F$ of positive finite measure which satisfies a $\mathbf{Renyi}$ $\mathbf{inequality}$; i.e., there is some constant $M$ such that
\begin{equation}\label{e:renyi}
\displaystyle \int_F (S_n(1_F))^2 dm \leq M \left( \int_F S_n(1_F) dm \right)^2 
\end{equation}
for every $n \in \mathbb{N}$. If this inequality holds only on a subset $\{n_i\} \subset \mathbb{N}$, we say that $T$ is $\mathbf{subsequence}$ $\mathbf{rationally}$ $\mathbf{ergodic}$. Some authors adopt this as the definition of rational ergodicity instead (see e.g. \cite{Ci11}). It was shown in \cite{Aa77} that rational ergodicity implies weak rational ergodicity.  It is not currently known whether these notions are equivalent. \\
\\
We say that $T$ is $\mathbf{boundedly \ rationally \ ergodic}$ (see \cite{Aa79}) if there exists a set $F$ of positive finite measure such that
\[
\displaystyle \sup_{n \geq 1} \left\| \dfrac{1}{a_n(F)} S_n(1_F) \right\|_{\infty} < \infty.
\]
In \cite{Aa79}, it was shown that bounded rational ergodicity is a strictly stronger property than rational ergodicity. It is not difficult to see that the proof of Theorem~\ref{T:wraterg} (in particular, the establishment of (5) for all $B_1$) yields bounded rational ergodicity for the transformations in question. Indeed, set $A = I$ in (5). Then there is a constant $c$ such that for any $n$ and $B_1 \subset I$, 
\[
\int_{B_1} \dfrac{1}{a_n(I)} S_n(1_I) dm = \dfrac{1}{a_n(I)} \sum_{k = 0}^{n-1} \mu(I \cap T^kB_1) \leq c \mu(B_1).
\]
This means that the average value of $S_n(1_I)/a_n(I)$ on $B_1$ is bounded above by $c$. Since this holds for every $B_1$, the essential supremum of $S_n(1_I)/a_n(I)$ must also be bounded above by $c$. Hence $T$ is boundedly rationally ergodic. \\
\\
Aaronson proved in \cite{Aa79} that every dyadic tower over the adding machine is boundedly rationally ergodic; Theorem~\ref{T:wraterg} extends this result to a larger class of transformations and uses a different approach. Some interesting examples of exponentially growing rank-one transformations with a bounded number of cuts include:

\begin{enumerate}[(a)]
\item Hajian-Kakutani skyscraper-type constructions \cite{HaKa70}:
\subitem $r_n = 2$, $\{s_{n, 0} = 0, s_{n, 1} \geq 2h_n\}.$

\noindent  (When $s_{n, 1} = 2h_n+1$ the transformation is spectral weakly mixing, see \cite{AdFrSi97}).
\item Chac\'on-like constructions: 
\subitem $r_n = 3$, $\{s_{n, 0} = 0, s_{n, 1} = 1, s_{n, 2} \geq 3h_n + 1\}$.

\noindent (When $s_{n, 2} = 3h_n + 1$ the transformation has infinite ergodic index, see \cite{AdFrSi97}, but is not power weakly mixing, see \cite{GHPSW03}.)
\end{enumerate}
\noindent
\\
We now prove a slightly different version of Theorem~\ref{T:wraterg} without the hypothesis of a bounded number of cuts but obtain the conclusion only a a subsequence.
\begin{theorem}\label{T:renyi}
Let $T$ be an exponentially growing rank-one transformation. Then $T$ is subsequence rationally ergodic on $F = I = (0, 1)$ along the sequence $\{n_m = M_m + 1\}$.
\end{theorem}

\begin{proof}
We verify the Renyi inequality \eqref{e:renyi} with $n = M_m+1$ and $M = 2$. Let $D(I, m) = \{I_j\}$ be the descendants of $I$ in column $C_m$, and set $N = |D(I, m)|$. Order $\{I_j\}$ by height of appearance in $C_m$ so that $I_1$ is the lowermost level of $\{I_j\}$ in $C_m$ and $I_N$ is the uppermost. Denote the heights of $\{I_j\}$ in $C_m$ by $\{h(I_j)\}$. \\
\\
Now, $S_n(1_I)(x)$ is equal to the number of $k$ with $0 \leq k \leq n-1$ such that $T^k(x) \in I$. Since $T$ is exponentially growing, this implies that $S_n(1_I)$ is constant on each $I_j$ and that the value of $S_n(1_I)$ on any fixed $I_l$ is the cardinality of the intersection $(h(I_l) + \{0, 1, \cdots, n-1\}) \cap \{h(I_j)\}$. (The relevant forward images $T^kI_l$ are simply upward translations.) On the other hand, it is obvious that $h(I_j) \in (h(I_l) + \{0, 1, \cdots, n-1\})$ exactly when $j \geq l$. Hence $S_n(1_I)$ takes the value $N + 1 - l$ on $I_l$. Restricting the domain of $S_n(1_I)$ to $I$, we thus have
\[
S_n(1_I) = \displaystyle \sum_{l=1}^{N} (N + 1 - l) 1_{I_l}.
\]
\noindent
Proving (7) is thus equivalent to showing 
\[
\displaystyle \sum_{l = 1}^N (N+1-l)^2w_m \leq 2 \left( \sum_{l=1}^N (N + 1 -l)w_m \right)^2.
\]
\noindent
Now, $w_m = 1/|D(I, m)| = 1/N$. Multiplying through by $N^2$ and reindexing yields the equivalent inequality
\[
\displaystyle N\left( \sum_{l = 1}^N l^2 \right) \leq 2 \left( \sum_{l=1}^Nl \right)^2.
\]
\noindent
The result then follows from the formulas for power sums.
\end{proof}


\section{Rational Weak Mixing}\label{S:ratweakmix}  
\noindent
In this section, we present a large class of transformations that are $not$ rationally weakly mixing. We obtain as a corollary the existence of transformations which are rationally ergodic and spectrally weakly mixing, but not rationally weakly mixing. \\
\\
We begin with an example of a rank-one transformation which is subsequence rationally weakly mixing. \\
\\
First, consider the Chac\'on rank-one transformation $T$ constructed by starting with the unit interval, cutting each column in half, and adding a single spacer on top of the right subcolumn at every step \cite{Ch69}. This transformation is finite measure-preserving and weakly mixing; thus, it is rationally weakly mixing. We claim that (in particular) $T$ is rationally weakly mixing on the unit interval $I = (0, 1)$. \\
\\
It is clear from the definition of weak rational ergodicity that if $T$ is weakly rationally ergodic on $F$, then $T$ is weakly rationally ergodic on any subset of $F$. Moreover, it was shown in \cite{Aa12} that for $T$ rationally weakly mixing, the class of sets $F$ satisfying \eqref{e:raterg} is the same as the class of sets $F$ satisfying \eqref{e:rwm}. This establishes the claim. \\
\\
Now let
\[
\phi_n(A, B) = \dfrac{1}{a_n(I)} \displaystyle \sum_{k = 0}^{n-1} |\mu(A \cap T^{k}B) - \mu(A)\mu(B)u_k(I)|
\]
\noindent
be the quotient from \eqref{e:rwm}, and let $D_m$ denote the collection of dyadic intervals of the form $(i/2^m, (i+1)/2^m)$ for $0 \leq i < 2^m$. Since $D_1$ is a finite collection and $T$ is rationally weakly mixing, there exists some natural number $m_1$ such that for all $A, B \in D_1$ we have $\phi_{m_1}(A ,B) < 1/2$. We claim that in fact this inequality is true for every rank-one transformation $\tilde{T}$ which shares its first $m_1$ stages of construction with $T$ (i.e., $\tilde{C}_n = C_n$ for all $n < m_1$). Indeed, for $A, B \subset I$, the value of $\phi_{m_1}(A ,B)$ depends only on the first $m_1$ stages of the construction of $T$, since the heights $D(I, m_1)$ are all less than $h_{m_1} - m_1$. \\
\\
We now define our desired transformation. We begin by following the construction of the transformation $T$ as described above, until we reach $C_{m_1}$. Then, at the $m_1$-th iteration, we add $2h_{m_1}$ spacers above the right subcolumn. Now, adding one spacer at each subsequent iteration gives another finite measure-preserving transformation, which is also weakly mixing. Hence, there is some $m_2 > m_1$ such that $\phi_{m_2}(A, B) < 1/4$ for all $A,B \in D_{1} \cup D_{2}$. \\
\\
We thus continue adding a single spacer at each step until we reach $C_{m_2}$, at which point we add $2h_{m_2}$ spacers. Proceeding inductively in this manner, we obtain a cutting-and-stacking transformation $T$ and a sequence $\{m_i\}$ such that for each $i$, $\phi_{m_i}(A, B) < 1/2^i$ for all $A, B \in D_1 \cup D_2 \cup \cdots \cup D_i$. The result is an invertible, infinite measure-preserving transformation which is rationally weakly mixing along $\{m_i\}$ for dyadic intervals. \\
\\
In order to extend to all subsets of $I$, we use  the following result due to Aaronson \cite{Aa12}:

\begin{lemma}
Let $T$ be an invertible measure-preserving transformation on a Polish space $X$, and assume that $T$ is rationally ergodic on some open set $F$. Suppose there is a countable base $\mathcal{C}$ for the topology of $F$ such that for every finite subcollection $\{C_i\} \subset \mathcal{C}$, there exists a finite subcollection $\{D_i\} \subset \mathcal{C}$ which is disjoint and has the same union. Then to establish rational weak mixing, it suffices to prove condition \eqref{e:rwm} for elements of $\mathcal{C}$.
\end{lemma}

\noindent
Lemma 3.1 also holds for establishing subsequence rational weak mixing, so long as rational ergodicity is known along the same subsequence. Since the transformation $T$ above may expressed as a dyadic tower over the adding machine, $T$ is rationally ergodic \cite{Aa79}. It follows from Lemma 3.1 that $T$ is subsequence rationally weakly mixing. \\
\\
\noindent
We now present a large class of examples that are not rationally weakly mixing. It will be convenient to write
\[
u_k(A, B) = \dfrac{\mu(A \cap T^kB)}{\mu(A)\mu(B)}
\]
so that for $A$, $B$ of positive measure, we can divide \eqref{e:rwm} by $\mu(A)\mu(B)$ to obtain 
\[
\dfrac{1}{a_n(F)} \displaystyle \sum_{k = 0}^{n-1} |u_k(A,B) - u_k(F)| \rightarrow 0.
\]
It is not difficult to see that in this case we must have $a_n(F)/a_n(A,B) \rightarrow 1$ \cite{Aa12}. This yields the following theorem:

\begin{theorem}\label{T:notrwm}
Let $T$ be a rank-one transformation constructed by cutting $C_n$ in half and adding at least $c_n \geq 2h_n$ spacers on top of the right subcolumn at every step. Then $T$ is not rationally weakly mixing.
\end{theorem}

\begin{proof}
We prove by contradiction. Suppose that $T$ is rationally weakly mixing on some set $F$. Choose a level $J$ which is at least $(3/4)$-full of $F$, and let $J_1$ and $J_2$ be the left and right halves of $J$. 
By applying $T^{-1}$ to $F$, we may assume that $J$ is the bottom level of some column $C_j$. Now, both $J_1$ and $J_2$ intersect $F$ in positive measure, so
\[
\displaystyle \dfrac{1}{a_n(F)} \displaystyle \sum_{k = 0}^{n-1} \left| u_k(J_1 \cap F) - u_k(F) \right| \rightarrow 0
\]
and
\[
\displaystyle \dfrac{1}{a_n(F)} \displaystyle \sum_{k = 0}^{n-1} \left| u_k(J_1 \cap F, J_2 \cap F) - u_k(F) \right| \rightarrow 0.
\]
Moreover, $a_n(F)/a_n(J_1 \cap F) \rightarrow 1$. Multiplying through by this limit and using the triangle inequality, we obtain
\begin{equation}
\displaystyle \dfrac{1}{a_n(J_1 \cap F)} \displaystyle \sum_{k = 0}^{n-1} \left| u_k(J_1 \cap F, J_2 \cap F) -  u_k(J_1 \cap F) \right| \rightarrow 0.
\end{equation}
Now, fix $k$ and suppose that $u_k(J_1 \cap F) > 0$. Then $\mu(J_1 \cap T^kJ_1) > 0$, so for sufficiently large $N$ we have $k \in D(J_1, N) - D(J_1, N)$. Similarly, if $u_k(J_1 \cap F, J_2 \cap F) > 0$, then $\mu(J_1 \cap T^{k + h_j}J_1) = \mu(J_1 \cap T^kJ_2) > 0$, which implies that $k + h_j \in D(J_1, N) - D(J_1, N)$. Hence we cannot have both $u_k(J_1 \cap F)$ and $u_k(J_1 \cap F, J_2 \cap F)$ nonzero, since then we would have $h_j \in (D(J_1, N) - D(J_1, N)) - (D(J_1, N) - D(J_1, N))$. As $D(J_1, N) = \{0, h_{j+1}\} \oplus \{0, h_{j+2}\} \oplus \cdots \oplus \{0, h_{N-1}\}$, this is easily seen to be impossible (given the fact that $c_n \geq 2h_n$ for all $n$). \\
\\
It is then immediate that
\[
\left| u_k(J_1 \cap F, J_2 \cap F) -  u_k(J_1 \cap F) \right| \geq u_k(J_1 \cap F)
\]
for every $k$. Indeed, if $u_k(J_1 \cap F) = 0$ then we are done; otherwise, $u_k(J_1 \cap F, J_2 \cap F)$ is $0$. It follows that the quotient (8) is bounded below by 1, which is a contradiction. This shows that $T$ is not rationally weakly mixing.
\end{proof}
\noindent
In particular, we obtain the following:

\begin{corollary}Let $T$ be the transformation constructed by cutting each column $C_n$ in half and adding $2h_n + 1$ spacers on top of the right subcolumn.
Then $T$ is rationally ergodic and spectrally weakly mixing but not rationally weakly mixing.
\end{corollary}

\begin{proof}
 This transformation is rationally ergodic by Theorem~\ref{T:wraterg} and the discussion following (and also by \cite[Theorem 1]{Aa79}), and is spectrally weakly mixing by  \cite[Proposition 1.1]{AdFrSi97}. By Theorem~\ref{T:notrwm}, however, $T$ is not rationally weakly mixing.
\end{proof}

\noindent
This negatively answers a question of Aaronson's. (As noted in the introduction, this result was obtained independently by Aaronson in \cite[1.1]{Aa12b}.) \\
\\
\noindent
We now extend Theorem~\ref{T:notrwm} to other rank-one transformations. Define
\[
H = \displaystyle \bigcup_{j = 0}^\infty H_j \setminus \{0\}
\]
and observe that the elements of $H$ are increasing when listed in the obvious order. (Begin with successive elements of $H_0 \setminus \{0\}$, followed by successive elements of $H_1 \setminus \{0\}$, and so on.) We say that a rank-one transformation is $\mathbf{steep}$ if $t_{i + 1} \geq 4 t_i$ for every pair of successive $t_i, t_{i+1} \in H$. Clearly, the transformations of Theorem~\ref{T:notrwm} are steep. In general, such transformations can be constructed by adding an exponentially increasing number of spacers above successive subcolumns. \\
\\
Steep transformations satisfy several nice properties, chief among which is a linear independence condition that allows us to extend Theorem 3.2. Suppose we have a linear combination
\begin{equation}
\displaystyle \sum_{t \in H} c_t t = 0
\end{equation}
\noindent
with the coefficients $c_t \in \{-2, -1, 0, 1, 2\}$. Then it is easily seen that all the $c_t$ must be 0. Similarly, we also obtain a uniqueness condition that will be useful in Section 5: every integer $k$ has at most one representation 
\[
k = \sum_{t \in H} c_t t
\]
with the $c_t \in \{-1, 0, 1\}$. Altering the definition of steepness slightly yields stronger forms of these properties; for example, requiring $t_{i + 1} \geq 5 t_i$ results in uniqueness of representation with $c_t \in \{-2, -1, 0, 1, 2\}$.


\begin{theorem}
Let $T$ be a normal, steep rank-one transformation. Then $T$ is not rationally weakly mixing.
\end{theorem}
\begin{proof}
We sketch the proof and leave the details to the reader. As before, we proceed by contradiction. Suppose $T$ is rationally weakly mixing on $F$, and let $J$ be a level $(3/4)$-full of $F$. Without loss of generality, we may assume that $J$ is the bottom level of some column $C_j$. Now, there must exist at least two descendants $J_1$ and $J_2$ of $J$ in $C_{j+1}$ that have positive intersection with $F$. For these levels, we have $J_2 = T^dJ_1$ for some $d \in H_j - H_j$. It then suffices to show that $d$ cannot be contained in $(D(J_1, N) - D(J_1, N)) - (D(J_1, N) - D(J_1, N))$, which follows from the linear independence (9). 
\end{proof}


\section{Relation to Double Ergodicity}\label{S:doubleerg}
\noindent
In this section we show that rational weak mixing implies double ergodicity and present an example suggesting the converse implication is false. 
\\
\\
\noindent
We begin by proving that rational weak mixing on $F$ implies double ergodicity for subsets of $F$.

\begin{theorem}
Suppose that $T$ is rationally weakly mixing on $F$. Then $T$ is doubly ergodic for all $A,B \subset F$.
\end{theorem}

\begin{proof}
Let $A, B \subset F$, and fix $\delta > 0$ such that 
\[
\delta < \frac{1}{2}\min(\mu(A)^2, \mu(A)\mu(B)). 
\]

\noindent
Since $T$ is rationally weakly mixing on $F$, 
\[
\dfrac{1}{a_n(F)} \displaystyle \sum_{k = 0}^{n-1} |\mu(A \cap T^kA) - \mu(A)^2u_k(F)| \rightarrow 0
\]
\noindent
and
\[
\dfrac{1}{a_n(F)} \displaystyle \sum_{k = 0}^{n-1} |\mu(A \cap T^kB) - \mu(A)\mu(B)u_k(F)| \rightarrow 0.
\]
\noindent
Summing these together, we obtain (by contradiction) that there exists a positive integer $k$ for which $u_k(F)>0$ and
\[
|\mu(A\cap T^kA)-\mu(A)^2u_k(F)|+|\mu(A\cap T^kB)-\mu(A)\mu(B)u_k(F)|<\delta u_k(F).
\]
\noindent
We thus have
\[
|\mu(A\cap T^kA)-\mu(A)^2u_k(F)|<\delta u_k(F)
\]
\noindent
and so
\[
\mu(A\cap T^kA)>u_k(F) (\mu(A)^2-\delta)>0
\]
for this $k$. Similarly, 
\[
\mu(A \cap T^kB)>u_k(F) (\mu(A)\mu(B)-\delta)>0.
\] 
By construction of $\delta$, this shows that $T$ is doubly ergodic on $F$.
\end{proof}

\noindent
We now extend this result to all of $X$. It was shown in \cite{Aa77} that if $T$ is weakly rationally ergodic on $F$, it is weakly rationally ergodic on any finite union $F_N = F \cup T(F) \cup \cdots \cup T^{N-1}(F)$. It follows that the analogous statement holds for rational weak mixing, giving the following theorem:

\begin{theorem}
Suppose that $T$ is rationally weakly mixing. Then $T$ is doubly ergodic.
\end{theorem}

\begin{proof}
Let $T$ be rationally weakly mixing on $F$, and suppose that $T$ is not doubly ergodic. Fix $A, B \subset X$ for which the double ergodicity condition fails; i.e., choose $A$ and $B$ such that for every $n$, either \linebreak $\mu(A \cap T^nA) = 0$ or $\mu(A \cap T^nB) = 0$. Since $F$ sweeps out $X$, there is some $N$ for which $F_N$ intersects both $A$ and $B$ in positive measure. Then $\tilde{A} = F_N \cap A$ and $\tilde{B} = F_N \cap B$ are sets of positive measure which fail the double ergodicity condition. But $T$ is doubly ergodic on $F_N$, a contradiction.
\end{proof}

\noindent
It is worth noting that (in general) the class of sets on which $T$ is doubly ergodic is $not$ a hereditary ring. For example, let $T$ be any doubly ergodic transformation on $X$, and define $S$ on $X \times \{0, 1\}$ by $S(x, 0) = (T(x), 1)$ and $S(x, 1) = (x, 0)$. Then $S$ is doubly ergodic on both $X \times \{0\}$ and $X \times \{1\}$, but not doubly ergodic on all of $X \times \{0, 1\}$. (Let $A = X \times \{0\}$ and $B = X \times \{1\}$.) \\
\\
We now investigate whether rational weak mixing is strictly stronger than double ergodicity. It will be useful for us consider transformations that are ``almost" steep. Recall that $T$ is steep if for any pair of successive elements $t_i, t_{i + 1}$ in $H = (H_0 \cup H_1 \cup \cdots) \setminus \{0\}$, we have $t_{i+1} \geq 4 t_i$. Now, suppose $T$ is constructed so that:
\begin{enumerate}[(a)]
\item Each column $C_n$ is cut into at least three subcolumns ($r_n \geq 3$).
\item We add zero spacers above the first subcolumn and one spacer above the second ($s_{n, 0} = 0$ and $s_{n, 1} = 1$).
\item We add a sufficient number of spacers above each subsequent subcolumn so that 
\[
\displaystyle \sum_{k = 0}^i h_{n, k} \geq 4 \left( \sum_{k = 0}^{i-1} h_{n, k} \right)
\]
for every $2 \leq i \leq r_n -1$.
\end{enumerate}
Then $T$ is ``almost" steep, in the sense that $t_{i+1} < 4t_i$ only when $t_i$ and $t_{i+1}$ are the first two nonzero elements of some $H_n$. For such $T$, we can still extract a (slightly technical) algebraic uniqueness condition in the spirit of (9). Indeed, let
\[
B_n = \{h_{n, 0}, h_{n, 0} + h_{n,1}\} \times \{h_{n, 0}, h_{n, 0} + h_{n,1}\} 
\]
and define
\[
A_n = (H_n \times H_n) \setminus (\Delta H_n \cup B_n).
\]
(Here, $\Delta H_n = \{(x, x) : x \in H_n\}$.) Then for any $(a, b), (a', b') \in A_n$ and $-M_n \leq k, k' \leq M_n$, the equality
\begin{equation}
k + a - b = k' + a' - b'
\end{equation}
implies 
\[
a = a', b = b', k = k'.
\]
\noindent
(The proof of this is not difficult and is left to the reader.) Before we proceed, it will be useful to establish following lemma:
\begin{lemma}
Let $J$ be any level, and fix $N$ sufficiently large. Suppose $(a, b) \in A_N$ and $-M_N \leq k \leq M_N$. Then
\[
\mu(J\cap T^{k + a - b}J) = \dfrac{1}{r_N} \mu(J\cap T^{k}J).
\]
\end{lemma}
\begin{proof}
By Lemma~\ref{L:lem2.1}, we have
\[
\mu(J\cap T^{k}J) = w_N\cdot |D(J, N) \cap (k + D(J, N))|
\]
and
\[
\mu(J\cap T^{k + a - b}J) = w_{N+1}|D(J, N+1) \cap (k + a - b + D(J, N+1))|.
\]
By uniqueness of (10), every representation of $k + a - b$ as an element of $D(J, N+1) - D(J, N+1)$ corresponds to exactly one representation of $k$ as an element of $D(J, N) - D(J, N)$, and vice-versa. Hence
\[
\mu(J\cap T^{k + a - b}J) = \dfrac{w_{N+1}}{w_N} \mu(J_1\cap T^{k}J) = \dfrac{1}{r_N} \mu(J\cap T^{k}J),
\] 
as desired.
\end{proof}
\noindent
We now show that if $T$ is almost steep and $\{r_n\}$ is sufficiently large, $T$ cannot be rationally weakly mixing.
 
\begin{theorem}\label{T:almoststeep}
Let $T$ be a rank-one transformation. Suppose that $T$ is almost steep (as described above), and that
\[
\displaystyle \sum_{n = 0}^{\infty} \dfrac{1}{r_n} < \infty.
\]
Then $T$ is not rationally weakly mixing.
\end{theorem}
\begin{proof}
We begin by proving that $T$ is not rationally weakly mixing on levels. Let $J$ be the bottom level of any column $C_j$, and let $J_1$ and $J_2$ be any two descendants of $J$ in $C_{j+1}$. Then $J_1 = T^dJ_2$ for some $d \in H_j - H_j$. As in Theorem~\ref{T:notrwm}, it suffices to disprove the convergence
\begin{equation}
\displaystyle \dfrac{1}{a_n(J_1)} \displaystyle \sum_{k = 0}^{n-1} \left|u_k(J_1) - u_k(J_1, J_2) \right| \rightarrow 0.
\end{equation}
To do this, define
\[
P_m = \sum\limits_{k = -M_m}^{M_m}|\mu(J_1\cap T^kJ_1)-\mu(J_1\cap T^{k+d}J_1)|
\]
and
\[
Q_m = \sum\limits_{k = -M_m}^{M_m}\mu(J_1\cap T^kJ_1).
\]
For $m$ sufficiently large, $R_m = P_m/ Q_m$ approximates the quotient (11), so it is enough to show that $R_m$ is bounded below by some positive constant. \\
\\
\noindent
Any choice of $(a, b) \in A_m$ and $- M_m \leq k \leq M_m$ yields a unique number $k + a - b$ between $-M_{m+1}$ and $M_{m+1}$. Hence 
\begin{align*}
P_{m+1} &= \sum\limits_{k = -M_{m+1}}^{M_{m+1}}|\mu(J_1\cap T^kJ_1)-\mu(J_1\cap T^{k+d}J_1)| \\
&\geq \sum_{(a,b) \in A_m} \sum\limits_{k = -M_m}^{M_m}|\mu(J_1\cap T^{k + a - b}J_1)-\mu(J_1\cap T^{k+a-b+d}J_1)| \\
&= \dfrac{1}{r_m} \left(\sum_{(a,b) \in A_m}\sum\limits_{k = -M_m}^{M_m}|\mu(J_1\cap T^{k}J_1)-\mu(J_1\cap T^{k+d}J_1)|\right) \\
&= \dfrac{|A_m|}{r_m} P_m.
\end{align*}
\noindent
Moreover, the same argument as in Theorem~\ref{T:wraterg} shows
\[
Q_m = \sum\limits_{k = -M_m}^{M_m}\mu(J_1\cap T^kJ_1) = |D(J_1, m)|\mu(J_1)
\]
from which it follows that
\[
Q_{m+1} = r_mQ_m.
\]
We thus obtain
\[
R_{m+1} \geq \dfrac{|A_m|}{r_m^2} R_m.
\]
Now, $|A_m| = r_m^2 - r_m - 2$, so $R_m$ is bounded below by 
\[
\displaystyle \prod_{k=0}^{\infty} \left(1 - \dfrac{1}{r_k} - \dfrac{2}{r_k^2}\right)R_0
\]
which is a positive constant by the hypotheses of the theorem. This bounds (10) from below along the sequence $\{M_m + 1\}$. Since $T$ is rationally ergodic along the same sequence by Theorem~\ref{T:renyi}, it follows that $T$ is not rationally weakly mixing. 
\end{proof}

\noindent
We now show that $T$ is doubly ergodic for levels, suggesting that rational weak mixing is strictly stronger than double ergodicity.

\begin{lemma}
The transformation $T$ above is doubly ergodic for levels.
\end{lemma}

\begin{proof}
We check that for any pair of levels $A$ and $B$, there exists an integer $n$ such that both $\mu(A\cap T^nA)>0$ and $\mu(B\cap T^nA)>0$. Without loss of generality, we may assume that $A$ is the bottom level of some column $C_j$ and that $B = T^dA$. It then suffices to prove there exists an $n$ such that both $n$ and $n+d$ are in $D(A,N)-D(A,N)$ (for $N$ sufficiently large). This is easy; simply choose
\[
n = h_{j+1,0}+\cdots+h_{j+d,0}.  
\]
Then 
\[
n + d = ((2h_{j+1, 0} + 1) +\cdots + (2h_{j+d,0} + 1)) - (h_{j+1,0}+\cdots+h_{j+d,0}),
\]
as desired.
\end{proof}


\section{Independence from Zero-type}\label{S:zerotype} 
\noindent
We now show that (subsequence) rational weak mixing and zero-type are independent (i.e., do not imply each other).
We say that $T$ is $\mathbf{zero}$-$\mathbf{type}$ if $\mu(A \cap T^nA) \rightarrow 0$ for all sets $A$ of finite measure \cite{HaKa64}. It is well-known that in order to show a conservative ergodic transformation is zero-type, it suffices to check this convergence for a single set $A$ of positive finite measure \cite{HaKa64}. We show that every steep transformation with an increasing number of cuts is zero-type. 

\begin{theorem}
Let $T$ be a normal,  steep rank-one transformation, and suppose that $\{r_n\}$ is nondecreasing with $\sup\{r_n\} = \infty$. Then $T$ is zero-type.
\end{theorem}
\begin{proof}
Consider $I = (0, 1)$. For $N$ sufficiently large, we have
\[
\mu(I \cap T^k I) = \dfrac{|D(I, N) \cap (k + D(I, N))|}{|D(I, N)|}.
\]
Now, $|D(I, N) \cap (k + D(I, N))|$ counts the number of representations
\begin{equation}
k = \displaystyle \sum_{i = 0}^{N-1} (d_i - d_i')
\end{equation}
with $d_i, d'_i \in H_i$. (Recall that $D(I, N) = H_0 \oplus H_1 \oplus \cdots \oplus H_{N-1}$.) If $k \notin D(I, N) - D(I, N)$, then $\mu(I \cap T^k I) = 0$, so suppose that $k \in D(I, N) - D(I, N)$. Then there is at least one representation
\begin{equation}
k = \displaystyle \sum_{i = 0}^{N-1} (x_i - x_i')
\end{equation}
with $x_i, x'_i \in H_i$. Now, fix $n$ and suppose $x_n - x'_n \neq 0$. By uniqueness of representation, any other representation (12) of $k$ must have $d_n = x_n$ and $d'_n = x'_n$. In particular, the only indices $i$ at which (12) can differ from (13) are those for which $x_i - x'_i = 0$. In these cases we must have $d_i = d'_i$, but otherwise there are no restrictions (i.e., $d_i = d'_i$ can be any element of $H_i$). Hence
\[
\displaystyle |D(I, N) \cap (k + D(I, N))| = \prod \limits_{x_i - x'_i = 0} |H_i|
\]
with the product being taken over all $i$ for which $x_i - x'_i = 0$. Since
\[
\displaystyle |D(I, N)| = \prod \limits_{i = 0}^{N-1} |H_i|
\]
it follows that
\[
\mu(I \cap T^k I) = \left( \prod \limits_{x_i - x'_i \neq 0} |H_i| \right)^{-1}.
\]
Now, if $k > M_n$, then the representation (13) of $k$ must have $x_m - x'_m \neq 0$ for some $m \geq n$. This implies that
\[
\mu(I \cap T^k I) \leq \dfrac{1}{|H_m|} = \dfrac{1}{r_m} \leq \dfrac{1}{r_n},
\]
which shows $\mu(I \cap T^k I) \rightarrow 0$ as $k \rightarrow \infty$. Hence $T$ is zero-type, as desired. 
\end{proof}
\noindent
We thus have:
\begin{theorem}
There exist rank-one transformations that are zero-type but not rationally weakly mixing.
\end{theorem}
\noindent
In \cite{Aa12b}, Aaronson recently constructed a zero-type transformation of the form $T\times S$, where $S$ is a Markov shift, such that $T\times S$ is not subsequence rationally weakly mixing.
 Our examples, however, are rank-one, so of a different nature, and  were constructed independently.\\
\\
\noindent
We note that it follows from Theorem F in Aaronson ~\cite{Aa12} that there exist subsequence rationally weakly mixing transformation of  positive type; a rank-one example is given by the subsequence rationally weakly mixing transformation of Section 3. (Indeed, this is partially rigid since $\mu(I \cap T^{h_i}I) \geq 1/2$ for every $i$.) Aaronson ~\cite{Aa12b} also constructed 
positive-type, rank-one, transformations that are not subsequence rationally ergodic.


\section{Examples of Rational Weak Mixing}\label{S:ratweakmixex}
\noindent
We end with a construction of a positive-type rank-one transformation which is rationally weakly mixing. Let $T$ be a Chac\'on-like transformation ($r_n = 3$, $\{s_{n, 0} = 0, s_{n, 1} = 1, s_{n, 2} \geq 3h_n + 1\}$) with enough spacers added above every third subcolumn so as to have $h_{n+1} = 3^c h_n$ for some fixed integer $c \geq 2$. Then $h_n = 3^{cn}$ and $D(I, n) = \displaystyle H_0 \oplus H_1 \oplus \cdots \oplus H_{n-1}$, where $H_i = \{0,h_i,2h_i +1\}$. 

\begin{theorem}\label{T:ratwm}
The above transformation $T$ is rationally weakly mixing.
\end{theorem}

\begin{proof}
We prove rational weak mixing for levels. Let $j\in \mathbb N_0$,  $J = J_1$ be the bottom level of $C_j$, and let $J_2 = T^dJ_1$. As in the proof of Theorem~\ref{T:almoststeep}, it suffices to show the convergence (11). Now, for any $n$, we may choose $m$ such that $M_{m-1} \leq n-1 < M_m$. Then the quotient (11) is asymptotically bounded  above by $P_m / Q_{m-1} = 3 P_m / Q_m$, so it suffices to prove $P_m / Q_m \rightarrow 0$ as $m \rightarrow \infty$. \\
\\
By the triangle inequality,
\begin{align*}
P_m &= \displaystyle\sum_{k = -M_m}^{M_m}|\mu(J \cap T^k J) - m (J \cap T^{k+d}J)| \\
&\leq \displaystyle \sum_{\ell=0}^{d-1}\displaystyle\sum_{k = -M_m}^{M_m} |\mu(J \cap T^{k+\ell} J) - m (J \cap T^{k+1+\ell}J)|.
\end{align*}
\noindent
Since each of the $d$ outer sums on the right differs from the $\ell=0$ sum by a finite number of terms, it suffices to prove the convergence with only the $\ell= 0$ sum. That is, we wish to show
\[
\frac{\sum\limits_{k = -M_m}^{M_m}|\mu(J\cap T^kJ)-\mu(J\cap T^{k+1}J)|}{\sum\limits_{k = -M_m}^{M_m}\mu(J\cap T^kJ)} \rightarrow 0.
\]
\noindent
To do this, it will be useful to introduce some auxiliary functions. Given $(d, d') \in D(J, m) \times D(J, m)$, write
\begin{equation}
d - d' = \displaystyle \left(\sum_{i=j}^{m - 1} d_i\right) - \left(\sum_{i=j}^{m - 1} d_i'\right)
\end{equation}
with each $d_i, d_i' \in \{0, h_i, 2h_i + 1\}$. Replacing each instance of $2h_i + 1$ with $2h_i$ in (14) yields a sum of the form
\begin{equation}
\displaystyle \sum_{i = j}^{m-1} \varepsilon_i3^{ci}
\end{equation}
with each $\varepsilon_i \in \{-2, -1, 0, 1, 2\}$. This defines a function
\[
g: D(J, m) \times D(J, m) \rightarrow \{-2, -1, 0, 1, 2\}^{m - j}
\]
taking the pair $(d, d')$ to the vector $\varepsilon = (\varepsilon_i)_{i = j}^{m-1}$, with the $\varepsilon_i$ as in (15). \\
\\
For each $\varepsilon \in \{-2, -1, 0, 1, 2\}^{m - j}$, define the ``multiplicity function" $\tilde{\varepsilon}$ on $D(J, m) - D(J, m)$ by
\[
\tilde{\varepsilon}(k) = |g^{-1} (\varepsilon) \cap \{(d,d') | d,d' \in D(J,m) \text{ and } d-d'=k \}|.
\]
That is, $\tilde{\varepsilon}(k)$ counts the number of pairs $(d,d')$ in $g^{-1}(\varepsilon)$ with $d-d'=k$. Then

\begin{lemma} The following properties hold: 
\begin{enumerate}[(a)]
\item Fix $k \in D(J, m) - D(J, m)$. Then
\[
\displaystyle \sum_{\varepsilon} \tilde{\varepsilon}(k) = |D(J,m)\cap (k + D(J, m))|.
\] 
where the sum on the left is taken over all $\varepsilon \in \{-2, -1, 0, 1, 2\}^{m - j}$.
\item Fix $\varepsilon = (\varepsilon_i)_{i = j}^{m-1} \in \{-2, -1, 0, 1, 2\}^{m - j}$. For each $p \in \{-2, -1, 0, 1, 2\}$, let $a_p$ be the number of $\varepsilon_i$ equal to $p$. Then 
\[
\displaystyle \sum_k \tilde{\varepsilon}(k)=3^{a_0}2^{a_1+a_{-1}}.
\] 
where the sum on the left is taken over all $k \in D(J, m) - D(J, m)$.
\end{enumerate}
\end{lemma}
\begin{proof} For (a), simply observe that both the left and right-hand expressions count the number of pairs $(d, d') \in D(J, m) \times D(J, m)$ for which $d-d'=k$. For (b), observe that the sum on the left counts the number of pairs $(d,d')$ whose associated vector is $\varepsilon$. Now, if $\varepsilon_i = 0$ in (15), then we must have $d_i = d_i'$ in (14), and there are three ways that this can happen. Similarly, if $\varepsilon_i = 1$, then either $d_i = h_i$ and $d_i' = 0$, or $d_i = 2h_i + 1$ and $d_i' = h_i$. Proceeding in this manner, a counting argument yields the desired equality.
\end{proof}
\noindent
($\textit{Proof of Theorem~\ref{T:ratwm}, continued.}$) \\
\\
Applying Lemma 6.2 (a) and Lemma 2.1, we thus need to show
\begin{equation}
\displaystyle \frac{\sum\limits_{k = -M_m}^{M_m}\left|\sum\limits_{\varepsilon}\tilde{\varepsilon}(k)-\sum\limits_{\varepsilon}\tilde{\varepsilon}(k+1)\right|}{\sum\limits_{k = -M_m}^{M_m}\sum\limits_{\varepsilon}\tilde{\varepsilon}(k)} \rightarrow 0.
\end{equation}
\noindent
By the triangle inequality, it suffices to prove this convergence after exchanging the order of summation in both the numerator and denominator. To this end, we exhibit a nonincreasing function $c(t)$ which converges to $0$ such that
\begin{equation}
\displaystyle R(\varepsilon):=\frac{\sum\limits_{k}|\tilde{\varepsilon}(k)-\tilde{\varepsilon}(k+1)|}{\sum\limits_{k} \tilde{\varepsilon}(k)}\leq c(a_1+a_{-1})
\end{equation}
\noindent
for each $\varepsilon$. Once we have such a function, we obtain the following bound for large enough $N$:

\begin{align*}
\displaystyle \sum_{\varepsilon}\sum\limits_{k = -M_m}^{M_m}|\tilde{\varepsilon}(k)-\tilde{\varepsilon}(k+1)| &\leq \sum_{\varepsilon}\sum\limits_{k = -M_m}^{M_m} \tilde{\varepsilon}(k) c(a_1 + a_{-1}) \\
&\leq c(N) \sum_{\substack{{\varepsilon} \text{ with} \\ {a_1 + a_{-1} \geq N}}}\sum\limits_{k = -M_m}^{M_m} \tilde{\varepsilon}(k)  \\
&+ c(0)\sum_{\substack{{\varepsilon}  \text{ with} \\ {a_1 + a_{-1} < N}}}\sum\limits_{k = -M_m}^{M_m} \tilde{\varepsilon}(k).
\end{align*}
\noindent
Dividing this by the denominator of (16) yields
\[
\displaystyle \frac{\sum\limits_{\varepsilon}\sum\limits_{k = -M_m}^{M_m}|\tilde{\varepsilon}(k)-\tilde{\varepsilon}(k+1)|}{\sum\limits_{\varepsilon}\sum\limits_{k = -M_m}^{M_m}\tilde{\varepsilon}(k)} \leq c(N) + c(0) d(N,m)
\]
where
\[
\displaystyle d(N, m)  = \dfrac{\sum\limits_{\substack{{\varepsilon}  \text{ with} \\ {a_1 + a_{-1} < N}}} \sum\limits_{k = -M_m}^{M_m} \tilde{\varepsilon}(k)  }{\sum\limits_{\varepsilon}\sum\limits_{k = -M_m}^{M_m}\tilde{\varepsilon}(k)}.
\]
\noindent
We claim that $d(N, m) \rightarrow 0$ as $m \rightarrow \infty$. Combined with the fact (still to be proven) that $c(N) \rightarrow 0$ as $N \rightarrow \infty$, this will imply the convergence of (16) to zero. \\
\\
By Lemma 6.2 (b),
\begin{align*}
d(N, m) &= \left( \sum\limits_{\substack{{\varepsilon}  \text{ with} \\ {a_1 + a_{-1} < N}}} 3^{a_0}2^{a_1+a_{-1}} \right) / \left( \sum\limits_{\varepsilon}3^{a_0}2^{a_1+a_{-1}} \right) \\
&\leq 2^N \left( \sum\limits_{\varepsilon} 3^{a_0} \right) / \left( \sum\limits_{\varepsilon}3^{a_0}2^{a_1+a_{-1}} \right)
\end{align*}
with the sums taken over all $\varepsilon \in \{-2, -1, 0, 1, 2\}^{m - j}$. Now, $\sum_{p=-2}^{2} a_p = m - j$ for each such $\varepsilon$, so we can view the sums over $\varepsilon$ as sums over 5-tuples $(a_0,a_{-1},a_{1},a_{-2},a_{2})$ of non-negative integers summing to $m-j$. That is, the above expression is equal to \\
\[
2^N \left(\sum \binom{m-j}{a_0, a_{-1}, a_{1}, a_{-2}, a_{2}} 3^{a_0} \right) / \left( \sum \binom{m-j}{a_0, a_{-1}, a_{1}, a_{-2}, a_{2}} 3^{a_0}2^{a_1}2^{a_{-1}} \right)
\]
\\
\noindent
with the summation as described above and $\binom{m-j}{a_0, a_{-1}, a_{1}, a_{-2}, a_{2}}$ the multinomial coefficient ``$m-j$ choose $a_0, \cdots, a_2$". By the identity
\[
(x_1 + x_2 + x_3 + x_4 + x_5)^n = \sum_{e_1 + e_2 + e_3 + e_4 + e_5 = n} \binom{n}{e_1, e_2, e_3, e_4, e_5} x_1^{e_1} \cdots x_5^{e_5},
\]
this is equal to
\[
2^N (3 + 1 + 1 + 1 + 1)^{m-j}/(3 + 2 + 2 + 1 + 1)^{m-j}.
\]
Hence $D(N, m) \leq 2^N (7/9)^{m-j}$, which clearly goes to zero as $m \rightarrow \infty$. \\
\\
Next, we will show that 
\[
s(n):=\sup_{\varepsilon} \left\lbrace R(\varepsilon): a_1+a_{-1}=n\right\rbrace
\]
converges to $0$. (See (17) for a definition of $R(\varepsilon)$.) Setting 
\[
c(t):=\sup\left\lbrace s(n): n\geq t\right\rbrace
\]
then produces a nonincreasing function with the desired properties and completes the proof. \\
\\
\noindent
Fix $\varepsilon$, and let $a$ be the minimum element of $D(J, m) - D(J, m)$ for which $\tilde{\varepsilon}(a) > 0$. Any $k \in D(J,m) - D(J,m)$ is expressible as $k = \sum \varepsilon_i 3^{ci} + \sum (+1) + \sum (-1)$, with the $+1$'s and $-1$'s coming from choosing $2h_i + 1$ for $d_i$ and $d_i'$ in (14). We now ask: how many $+1$'s and $-1$'s do we have for $k = a$? We have only one way of obtaining $\varepsilon_i = 2$ in (14): namely, $(2h_i + 1)  - 0$. Similarly, we only have one way of obtaining $\varepsilon_i = -2$: namely, $0 - (2h_i +1)$. This introduces $a_2$ number of $+1$'s and $a_{-2}$ number of $-1$'s.  We have three ways of obtaining $\varepsilon_i = 0$, none of which introduce a net number of  $+1$'s or $-1$'s. For $\varepsilon_i = 1$, we have two possibilities: either $h_i - 0$ or $(2h_i + 1) - h_i$. Since we want to minimize $a$, we choose the former. Similarly, for $\varepsilon_i = -1$, we must have either $0 - h_i$ or $h_i - (2h_i + 1)$, and to minimize $a$ we choose the latter. It thus follows that $a$ has $a_2$ number of $+1$'s and $a_{-2} + a_{-1}$ number of $-1$'s; moreover, $\tilde{\varepsilon}(a) = 3^{a_0}$. It is then not difficult to see that
\[
\tilde{\varepsilon}(a + k) = 3^{a_0} \binom{a_1 + a_{-1}}{k}
\] 
for all $0 \leq k \leq a_1 + a_{-1}$, and is $0$ otherwise. \\
\\
Letting $n = a_1 + a_{-1}$, we thus have
\[
R(\varepsilon)=\dfrac{1}{2^n} \left(\sum_{k = 0}^{n-1} \left|\binom{n}{k}-\binom{n}{k+1}\right| + 2\right).
\]
Suppose $n = 2l -1$. (The case when $n$ is even is dealt with similarly.) Since 
\[
\left|\binom{n}{k}-\binom{n}{k-1}\right|=\binom{n + 1}{k}\left|\frac{(n + 1)-2k}{n+1}\right|,
\]
the above expression yields
\begin{align*}
R(\varepsilon) &= \dfrac{1}{2^n} \left( \sum_{k=1}^{n}\binom{n + 1}{k}\left|\frac{(n + 1)-2k}{n+1}\right| + 2\right) \\
&= \dfrac{1}{2^n} \left(\sum_{k=1}^{2l - 1} \binom{2l}{k} \left| \dfrac{l - k}{l} \right| + 2\right) \\
&\leq \dfrac{1}{2^n}  \left(2 \sum_{k=0}^{l} \binom{2l}{k} \left(\dfrac{l - k}{l} \right)\right).
\end{align*}
\noindent
Using the combinatorial identity
\[
\sum_{k=0}^{l} \binom{2l}{k} \left( \frac{l-k}{l}  \right) = \frac{l+1}{2l}\binom{2l}{l+1},
\]
we obtain
\[
R(\varepsilon) \leq \dfrac{1}{2^{2l}} \binom{2l}{l+1}.
\]
It is not difficult to see that this goes to $0$ as a function of $l$, thus proving that $T$ is rationally weakly mixing for levels. By Theorem~\ref{T:wraterg} and Lemma 3.1, it follows that $T$ is rationally weakly mixing.
\end{proof}

}

\normalsize

\bibliographystyle{plain}
\bibliography{References}

\end{document}